\documentclass[pdflatex,sn-mathphys-num]{sn-jnl}

\usepackage{graphicx}%
\usepackage{multirow}%
\usepackage{amsmath,amssymb,amsfonts}%
\usepackage{amsthm}%
\usepackage{mathrsfs}%
\usepackage[title]{appendix}%
\usepackage{xcolor}%
\usepackage{textcomp}%
\usepackage{manyfoot}%
\usepackage{booktabs}%
\usepackage{algorithm}%
\usepackage{algorithmicx}%
\usepackage{algpseudocode}%
\usepackage{listings}%

\theoremstyle{thmstyleone}%
\newtheorem{theorem}{Theorem}
\newtheorem{proposition}[theorem]{Proposition}%
\newtheorem{lemma}{Lemma}

\theoremstyle{thmstyletwo}%

\theoremstyle{thmstylethree}%
\newtheorem{definition}{Definition}%

\begin{document}

\title[FPT $t$-hypergraphicality]{Fixed-Parameter Tractability of $t$-Uniform Hypergraphicality}


\author[1,4]{\fnm{Riley} \sur{Brown}}\email{rydb@proton.me}

\author*[2,3,4]{\fnm{Istv\'an} \sur{Mikl\'os}}\email{miklos.istvan@renyi.hu}\orcid{https://orcid.org/0000-0002-1160-6855}

\affil[1]{\orgname{University of Nebraska}, \orgaddress{\street{1400 R Street}, \city{Lincoln}, \postcode{NE 68588}, \state{Nebraska}, \country{United States}}}

\affil[2]{\orgdiv{Department of Stochastics}, \orgname{HUN-REN R\'enyi Institute}, \orgaddress{\street{Re\'altanoda u. 13-15}, \city{Budapest}, \postcode{1053}, \country{Hungary}}}

\affil[3]{\orgdiv{ILAB}, \orgname{HUN-REN SZTAKI}, \orgaddress{\street{L\'agym\'anyosi u. 11}, \city{Budapest}, \postcode{1111}, \country{Hungary}}}

\affil[4]{\orgname{Budapest Semesters in Mathematics}, \orgaddress{\street{Bethlen G\'abor t\'er 2}, \city{Budapest}, \postcode{1071}. \country{Hungary}}}


\abstract{We study the $t$-uniform hypergraphicality problem under a compressed representation of the degree sequence. Instead of listing all vertex degrees explicitly, the input consists of pairs
$$
(\delta_1,n_1),\dots,(\delta_k,n_k),
$$
meaning that exactly $n_i$ vertices have degree $\delta_i$. Thus the parameter $k$ denotes the number of distinct degrees.

Although deciding $t$-hypergraphicality is NP-complete for every fixed $t>2$, we prove that the problem is fixed-parameter tractable parameterized by $(k,t)$. Our result shows that tractability extends substantially beyond previously known bounded-range regimes: even degree sequences with large overall degree spread can be handled efficiently when the number of distinct degrees is bounded.

Our approach decomposes hyperedges according to their types with respect to the degree classes, yielding a bounded-dimension spectrum representation. Using balancing hinge-flips, we show that every feasible spectrum can be transformed into a realization of the prescribed degree sequence. This leads to an integer programming feasibility formulation with
$$
\binom{t+k-1}{k-1}
$$
variables. Applying Lenstra's theorem yields an FPT algorithm running in time
$$
f(k,t)\cdot \mathrm{poly}(L),
$$
where $L$ denotes the encoding length of the compressed input.}






\maketitle

\section{Introduction}\label{sec1}

The characterization of degree sequences is one of the classical topics of combinatorics. In the simple graph case, the Erd\H{o}s--Gallai theorem \cite{ErdosGallai} and the Havel--Hakimi algorithm \cite{Hakimi,Havel} give complete characterizations and efficient recognition algorithms for graphical degree sequences. In contrast, the analogous problem for uniform hypergraphs is substantially more difficult.

A sequence $d=(d_1,\dots,d_n)$ of nonnegative integers is called \emph{$t$-hypergraphical} if there exists a simple $t$-uniform hypergraph whose vertex degrees are given by $d$. Already for $t=3$, deciding hypergraphicality is NP-complete, as proved by Deza, Levin, Meesum, and Onn~\cite{Deza}. More recently, Mikl\'os, Ruszink\'o, and Zavalnij extended this hardness result to every fixed $t>2$~\cite{MRZ}. Thus, unlike the graph case, one cannot expect a complete polynomial-time characterization of $t$-hypergraphicality in general.

In the search for tractable subclasses of the hypergraphicality problem, most known positive results rely on some form of degree regularity. One natural direction is to restrict the range of the degrees. For example, nearly regular degree sequences --- where every degree equals either $d$ or $d+1$ --- admit substantially simpler realizability behavior. More generally, bounded-spread regimes have been investigated by Palma, Frosini, and Rinaldi~\cite{PFR}.
These results have been significanty extended in the recent manuscript of Logsdon, Maheswari, Mikl\'os, and Zhang~\cite{LMMZ}, where lower and upper bounds have been identified on the degrees in which $3$-hypergraphicality becomes tractable.

These tractability results are all based on restricting the degree range globally. It has been shown that wide degree ranges already lead back to NP-hardness in the unrestricted setting~\cite{LMMZ}. This raises the question whether tractability can still be recovered from weaker forms of regularity.

The present paper takes such an approach. Instead of controlling the global degree range, we exploit regularity arising from multiplicities of equal degrees. Indeed, each degree class is itself regular, even if the overall degree sequence is highly non-regular. Thus even highly non-regular degree sequences may still possess substantial internal structure provided that the number of distinct degrees is small.

This observation naturally leads to a parameterization by the number $k$ of distinct degrees together with the uniformity parameter $t$. Our main result shows that this compressed degree representation yields fixed-parameter tractability. Importantly, this tractability mechanism is fundamentally different from previously known bounded-range regimes. In particular, our algorithm applies to instances whose overall degree range may be substantially larger than the tractable intervals identified in~\cite{LMMZ}. Of course, unrestricted wide-range instances remain NP-hard~\cite{LMMZ,MRZ}. However, our results show that large degree ranges alone do not preclude tractability, provided that the number of distinct degrees is bounded.

Motivated by this perspective, we study the parameterized complexity of the $t$-hypergraphicality problem under a compressed representation of the degree sequence. Instead of listing all degrees individually, we group equal degrees together. Thus, the input consists of pairs
\[
(\delta_1,n_1),\dots,(\delta_k,n_k),
\]
where exactly $n_i$ vertices have degree $\delta_i$. The parameter $k$ denotes the number of distinct degrees, while $t$ is the uniformity.

This representation is significantly more concise than the explicit degree sequence representation when the number of distinct degrees is small. In particular, the input length is proportional to the number of bits needed to encode the integers $\delta_i$ and $n_i$, rather than to the total number of vertices. Consequently, fixed-parameter tractability in terms of $(k,t)$ is a meaningful and nontrivial notion.

Our main result is that $t$-hypergraphicality is fixed-parameter tractable parameterized by the pair $(k,t)$. The proof combines structural balancing arguments with an integer programming formulation of bounded dimension.

The key combinatorial tool is the \emph{balancing hinge-flip} operation. A hinge-flip replaces one hyperedge by another hyperedge sharing $t-1$ vertices. Such operations allow local redistribution of degrees while preserving uniformity. Variants of these balancing techniques already appeared in earlier work of Mikl\'os, Ruszink\'o, and Zavalnij~\cite{MRZ}. In the present paper, we use them to transform arbitrary realizations into realizations whose degrees are balanced inside prescribed vertex classes.

A second key concept is the decomposition of hyperedges according to their \emph{types}. After partitioning the vertex set into degree classes
\[
V = V_1 \cup \cdots \cup V_k,
\]
each hyperedge determines a vector recording how many vertices it contains from each class. The numbers of hyperedges of the various types form the \emph{spectrum vector} of the hypergraph. Crucially, the spectrum vector determines the total degree contribution inside every degree class. This leads to an integer feasibility formulation with
\[
\binom{t+k-1}{k-1}
\]
variables, depending only on the parameters $k$ and $t$.

Combining this formulation with classical fixed-dimension integer programming algorithms yields our main theorem: $t$-hypergraphicality is fixed-parameter tractable parameterized by $(k,t)$.

Our results demonstrate that although hypergraphicality is NP-hard in general, substantial algorithmic tractability emerges when the degree structure is sufficiently compressed.








\section{Preliminaries}


\begin{definition}
A \emph{$t$-uniform hypergraph} is a pair $H = (V,E)$ where $V$ is a finite vertex set and $E \subseteq \binom{V}{t}$. The degree of a vertex $v$ is the number of hyperedges $e$ such that $v\in e$. The degree of vertex $v$ is denoted by $d(v_i)$.
\end{definition}

\begin{definition}
Let $d = (d_1,\dots,d_n)$ be a sequence of nonnegative integers. We say that $d$ is \emph{$t$-hypergraphical} if there exists a $t$-uniform hypergraph $H = (V,E)$ such that $V = \{v_1, v_2,\ldots,v_n\}$ and for all $i=1,2,\ldots, n$ $d(v_i) = d_i$.
\end{definition}


Since we consider degree sequences in which there are $k$ different degrees,
we assume that the degree sequence is given in compressed form as follows. Let the distinct degrees be
\[
\delta_1, \delta_2, \dots, \delta_k,
\]
and let $n_i$ denote the number of vertices of degree $\delta_i$. Thus
\[
\sum_{i=1}^k n_i = n.
\]

\begin{definition}
The \emph{input size} $L$ is the total number of bits required to encode the integers
\[
(\delta_1,n_1), (\delta_2,n_2), \dots, (\delta_k,n_k).
\]
\end{definition}

In particular, $L = O\!\left(\sum_{i=1}^k (\log \delta_i + \log n_i)\right)$.


\begin{definition}
A problem is \emph{fixed-parameter tractable (FPT)} with parameter set $p$ if it can be solved in time
\[
f(p)\cdot \mathrm{poly}(L),
\]
where $L$ is the input size measured in bits.
\end{definition}

In this paper, the parameter set is $p:= (k,t)$, where $k$ is the number of distinct degrees, and $t$ is the uniformity.


The key tool of our FTP algorithm is the hinge-flip operation. Hinge-flip operations were first used in approximating the permanent \cite{jerrumsinclair}.  Hinge-flip operations were considered on simple graphs in several recent papers \cite{RechnerStrowickMullerHannemann,AmanatidisKleer2023,ErdosMezeiMiklos} and also on $t$-uniform hypergraphs \cite{LiMiklos,LMMZ,MRZ}.

\begin{definition}
Let $H=(V,E)$ be a $t$-uniform hypergraph. A \emph{hinge-flip} is an operation that removes a hyperedge $e \in E$ and adds a hyperedge $e' \notin E$ such that
\[
|e \cap e'| = t-1.
\]
\end{definition}

Observe that a hinge-flip replaces exactly one vertex of $e$ by another vertex, hence it changes the degrees of exactly two vertices: one decreases by $1$ and one increases by $1$.

\begin{definition}
A hinge-flip is called \emph{balancing} if it decreases the degree of a vertex with larger degree $d_j$ and increases the degree of a vertex with smaller degree $d_i$ such that $d_j > d_i+1$.
\end{definition}

Balancing hinge-flips will be used to redistribute degrees within vertex classes. The simple graph version of the following lemma was observed by Aigner and Triesh \cite{AignerTriesh}. It was also proved for $3$-uniform \cite{LiMiklos} and $t$-uniform \cite{MRZ} hypergraphs. For sake of completeness, we give the proof also here.

\begin{lemma}[Existence of balancing hinge-flips]
Let $H$ be a $t$-uniform hypergraph and let $u,v \in V$ satisfy $d(u) > d(v)+1$. Then there exists a hinge-flip that decreases $d(u)$ by $1$ and increases $d(v)$ by $1$.
\end{lemma}

\begin{proof}
Let
\[
\mathcal{N}(w) = \{ S \subseteq V \setminus \{w\} : |S| = t-1,\ S \cup \{w\} \in E \},
\]
so that $|\mathcal{N}(w)| = d(w)$.

Partition $\mathcal{N}(u)$ into three parts:
\begin{align*}
\mathcal{A}_1 &= \{ S \in \mathcal{N}(u) : v \in S \}, \\
\mathcal{A}_2 &= \{ S \in \mathcal{N}(u) : v \notin S,\ S \in \mathcal{N}(v) \}, \\
\mathcal{A}_3 &= \{ S \in \mathcal{N}(u) : v \notin S,\ S \notin \mathcal{N}(v) \}.
\end{align*}

Thus $\mathcal{N}(u) = \mathcal{A}_1 \cup \mathcal{A}_2 \cup \mathcal{A}_3$ is a disjoint union.

We claim that
\[
|\mathcal{A}_1| + |\mathcal{A}_2| \le d(v).
\]
Indeed, each $S \in \mathcal{A}_1$ corresponds to an edge containing both $u$ and $v$, and each $S \in \mathcal{A}_2$ corresponds to an edge $S \cup \{v\}$ containing $v$ but not $u$. These edges are all distinct, hence their total number is at most $d(v)$.

Therefore,
\[
|\mathcal{A}_3| = d(u) - (|\mathcal{A}_1| + |\mathcal{A}_2|) \ge d(u) - d(v) > 0.
\]

Thus there exists $S \in \mathcal{A}_3$. Then $S \cup \{u\} \in E$, $v \notin S$, and $S \cup \{v\} \notin E$. Replacing $S \cup \{u\}$ by $S \cup \{v\}$ gives the desired hinge-flip.
\end{proof}

\noindent
The existence of balancing hinge-flips allows us to redistribute degrees within a prescribed subset of vertices. The following lemma was proved in \cite{MRZ}, for sake of completeness, we present the proof here, too.

\begin{lemma}[Balancing on a subset]\label{lem:balance_subset}
Let $H=(V,E)$ be a $t$-uniform hypergraph and let $W \subseteq V$. Then there exists a $t$-uniform hypergraph $H'=(V,E')$ such that:
\begin{itemize}
\item $d_{H'}(v) = d_H(v)$ for all $v \in V \setminus W$,
\item the degrees of vertices in $W$ are almost regular in $H'$, and
\item
\[
\sum_{v \in W} d_{H'}(v) = \sum_{v \in W} d_H(v).
\]
\end{itemize}
\end{lemma}

\begin{proof}
Starting from $H$, we iteratively apply balancing hinge-flips on pairs of vertices $u,v \in W$ with $d(u) > d(v)$.

Such a hinge-flip decreases $d(u)$ by $1$ and increases $d(v)$ by $1$, hence it preserves the total degree on $W$, and does not affect degrees outside $W$.

Let
\[
\overline{d} = \frac{1}{|W|} \sum_{w \in W} d(w)
\]
denote the average degree on $W$, and let $a = \lfloor \overline{d} \rfloor$, $b = \lceil \overline{d} \rceil$.

Consider the potential function
\[
\Phi = \sum_{w \in W} \min\left\{ |d(w) - a|,\; |d(w) - b| \right\}.
\]

We claim that as long as the degrees on $W$ are not almost regular, there exist $u,v \in W$ such that a balancing hinge-flip strictly decreases $\Phi$.

Indeed, if the degrees are not almost regular, then there exists a vertex with degree at most $a-1$ or at least $b+1$.

If there exists $v \in W$ with $d(v) \le a-1$, then since the average is $\overline{d}$, there must exist $u \in W$ with $d(u) \ge a+1$. Applying a balancing hinge-flip from $u$ to $v$ decreases $|d(v)-a|$ and do not increase $ \min\left\{ |d(u) - a|,\; |d(u) - b| \right\}$, and hence decreases $\Phi$.

Similarly, if there exists $u \in W$ with $\deg(u) \ge b+1$, then there exists $v \in W$ with $\deg(v) \le b-1$, and a balancing hinge-flip from $u$ to $v$ decreases $\Phi$.

Thus, in every step, $\Phi$ strictly decreases. Since $\Phi$ is a nonnegative integer, the process terminates. At termination, all degrees in $W$ belong to $\{a,b\}$, i.e., they are almost regular.
\end{proof}

\section{FPT algorithm}
First, we observe that $t$-uniform hypergraphs can be unequivocally factored in the following way.
Let $V$ be partitioned into \emph{degree classes}
\[
V = V_1 \cup \cdots \cup V_k,
\]
where $|V_i| = n_i$ and each vertex in $V_i$ has degree $\delta_i$.

\begin{definition}
The \emph{type} of a hyperedge $e$ is the vector $x = (x_1,\dots,x_k)$ such that
\[
x_i = |e \cap V_i| \quad \text{for each } i.
\]
\end{definition}
\noindent It follows that $\sum_{i=1}^k x_i = t$.
Let $\mathcal{T}$ denote the set of all possible types. Then
\[
|\mathcal{T}| = \binom{t+k-1}{k-1}.
\]

\begin{definition}
For each type $x \in \mathcal{T}$, the \emph{factor} corresponding to $x$ is the subhypergraph consisting of all hyperedges of type $x$.
\end{definition}

Thus any $t$-uniform hypergraph induces a decomposition of its edge set into factors indexed by $\mathcal{T}$.
We encode this decomposition via the following vector.

\begin{definition}
The \emph{spectrum vector} of a $t$-uniform hypergraph is the vector
\[
y = (y_x)_{x \in \mathcal{T}} \in \mathbb{Z}_{\ge 0}^{\mathcal{T}},
\]
where $y_x$ denotes the number of hyperedges of type $x$.
\end{definition}

The spectrum vector determines the total degree within each vertex class.

\begin{proposition}\label{prop:spectrum_degrees}
Let $y$ be the spectrum vector of a $t$-uniform hypergraph. Then for each $i \in [k]$,
$$
\sum_{x \in \mathcal{T}} x_i \, y_x \;=\; \sum_{v \in V_i} d(v).  
$$

\end{proposition}

\begin{proof}
Each hyperedge of type $x$ contributes exactly $x_i$ incidences to vertices in $V_i$. Summing over all edges yields the identity.
\end{proof}

In particular, the spectrum vector determines the total degree in each class $V_i$, and hence is consistent with the degree sequence if and only if
\begin{equation}
\sum_{x \in \mathcal{T}} x_i \, y_x = n_i \delta_i
\quad \text{for all } i \in [k].
    \label{eq:degree-constraints}
\end{equation}

This representation together with the following feasibility constraints will serve as the basis of our integer programming formulation.
\begin{lemma}[Feasible sizes]\label{lem:feasible_sizes}
Fix a type $x=(x_1,\dots,x_k) \in \mathcal{T}$. Let
\[
M_x := \prod_{i=1}^k \binom{|V_i|}{x_i}.
\]
Then for every integer $m$ with
\[
0 \le m \le M_x,
\]
there exists a $t$-uniform hypergraph consisting only of edges of type $x$ with exactly $m$ edges.
\end{lemma}

\begin{proof}
There are exactly $M_x$ distinct hyperedges of type $x$. Hence for any $m \le M_x$, we may choose an arbitrary subset of $m$ such edges.

This directly yields a $t$-uniform hypergraph of type $x$ with $m$ edges.
\end{proof}



We are ready to fomulate the integer programming feasibility problem. Introduce a variable $y_x$ for each type $x \in \mathcal{T}$, where $y_x$ denotes the number of hyperedges of type $x$. The integer feasibility problem consist of $k$ equations in form represented in equation~\ref{eq:degree-constraints} as well as the following $2|\mathcal{T}|$ number of inequalities.
For each type $x \in \mathcal{T}$, consider the following pair of inequalities as lower and upper bounds on $y_x$.
\begin{equation}
0 \le y_x \le \prod_{i=1}^k {n_i\choose x_i}. \label{eq:capacity-constraints}    
\end{equation}
Therefore, the integer programming feasibility problem consist of $k$ equations on the degree constraints (also expressable as $2k$ inequalities) and $2|\mathcal{T}| = 2{t+k-1\choose k-1}$ inequalites on the feasable sizes of $y_x$. The number of variables satisfies:
\[
|\mathcal{T}| = \binom{t+k-1}{k-1},
\]
which depends only on $(k,t)$.
Thus the integer feasibility problem has fixed dimension, and even further, a fixed number of inequalities in which only the coefficients depends on the input.

The key lemma of our FPT algorithm is the following.

\begin{lemma}\label{lem:ip_equivalence}
The degree sequence is $t$-hypergraphical if and only if the integer system expressed in equations~\ref{eq:degree-constraints}~and~\ref{eq:capacity-constraints} has a feasible solution.
\end{lemma}

\begin{proof}
($\Rightarrow$)
Let $H$ be a realization. For each type $x$, let $y_x$ be the number of edges of type $x$ in $H$. Then the degree constraints hold by counting incidences between edges and vertex classes, and clearly $0 \le y_x \le M_x$.

($\Leftarrow$)
Suppose $(y_x)_{x \in \mathcal{T}}$ is a feasible solution. For each type $x$, by Lemma~\ref{lem:feasible_sizes}, we can construct a hypergraph $H_x$ consisting of exactly $y_x$ edges of type $x$. Let $H$ be the union of all $H_x$. Then $H$ is $t$-uniform,simple and has spectrum $y$.

By Proposition~\ref{prop:spectrum_degrees}, for each $i \in [k]$,
\[
\sum_{v \in V_i} \deg_H(v) = \sum_{x \in \mathcal{T}} x_i y_x = n_i \delta_i.
\]

Thus, for each class $V_i$, the average degree is exactly $\delta_i$. Starting from $H$, we now apply Lemma~\ref{lem:balance_subset} with $W = V_i$ for each $i \in [k]$. This preserves degrees outside $V_i$, and transforms the degrees inside $V_i$ into an almost regular sequence with the same total sum $n_i \delta_i$. Since the average is an integer, the resulting degrees on $V_i$ are all equal to $\delta_i$.

Performing this procedure successively for all classes $V_1,\dots,V_k$, we obtain a $t$-uniform hypergraph realizing the prescribed degree sequence.
\end{proof}


We are ready to prove the main theorem.
\begin{theorem}
$t$-uniform hypergraphicality is FPT parameterized by $(k,t)$.
\end{theorem}

\begin{proof}
For each input 
$$
(n_1, \delta_1), (n_2,\delta_2),\ldots, (n_k,\delta_k),
$$
there is a $t$-uniform hypergraph with degrees prescribed in the input if and only if the integer programming feasibility problem expressed in equations~\ref{eq:degree-constraints}~and~\ref{eq:capacity-constraints} has a solution, according to Lemma~\ref{lem:ip_equivalence}. This is an integer feasibility problem with fixed number of variables.
By Lenstra's theorem~\cite{Lenstra}, integer programming in fixed dimension can be solved in time $f(k,t)\cdot \mathrm{poly}(L)$, where $L$ is the encoding length of the input.  The coefficients in the integer programming feasibility problem are either in form $n_i \delta_i$ which has $O(\log(n_i)+\log(\delta_i))$ digits or in form $\prod_{i=1}^k {n_i\choose x_i}$. Since each $x_i \le k$, $k$ is a constant, and for all $n_i$ and $x_i$, ${n_i\choose x_i} \le n_i^{x_i}$, thus, ${n_i\choose x_i} \le n_i^k$ each such coefficient has  $O\left(\sum_{i=1}^k \log(n_i)\right)$ digits.

Since all coefficients are polynomially bounded in the input size, and the number of coefficients is contant, the result follows.
\end{proof}

\backmatter

\bmhead{Acknowledgements}

The work was supported by the Hungarian NKFIH grant K132696 and by the European Union project RRF2.3.1-21-2022-00006 within the framework of Health Safety National Laboratory Grant no RRF-2.3.1-21-2022-00006.

\bibliography{sn-bibliography}

\end{document}